\documentclass[a4paper,12pt]{article}

\usepackage{typearea}
\typearea{15}

\usepackage{mymacros}

\usepackage{mathtools}
\usepackage{slashed}

\begin{document}
	\rightline{\baselineskip16pt\rm\vbox to20pt{
			{
				\hbox{OCU-PHYS-584}
				\hbox{AP-GR-193}
			}
			\vss}}%

	\begin{center}
		{\LARGE On Sasakian quasi-Killing spinors in three-dimensions}\\
		\bigskip\bigskip
		{\large
			Satsuki Matsuno\footnote{smatsuno43@gmail.com}\\
			Fumihiro Ueno\footnote{ueno-fumihiro-dx@alumni.osaka-u.ac.jp}
		}\\
		\bigskip
		{\it Department of Physics, Graduate School of Science,
			Osaka Metropolitan University\\
			3-3-138 Sugimoto, Sumiyoshi, Osaka 558-8585, Japan}
	\end{center}
	
	
	\begin{abstract}
		A Sasakian quasi-Killing spinor (SqK-spinor), which is a generalization of a Killing spinor on Sasakian manifolds, was defined in \cite{Kim Friedrich 2000}.
		The purpose of this paper is to study in detail SqK-spinors on three-dimensional pseudo-Riemannian Sasakian space-form.
		We briefly review some results on SqK-spinors and then investigate some geometric properties.
		First, we demonstrate that the Reeb vector field is described by a specific SqK-spinor. Then we establish that the motion of a charged particle in the presence of a contact Maxwell field can be described by a SqK-spinor. Furthermore, we discuss that almost all SqK-spinors provide solutions to the Einstein-Dirac system with a non-zero cosmological constant. 
		Additionally, we clarify that some SqK-spinor and contact Maxwell field constitute solutions to the Einstein-Dirac-Maxwell systems.
		Finally, we show explicit formulae of SqK-spinors in terms of elementary functions with respect to a certain frame.
	\end{abstract}
	
	\tableofcontents

	\section{Introduction}
	A Sasakian quasi-Killing spinor (SqK-spinor) has significance from some perspective.
	In \cite{Kim Friedrich 2000} a weak Killing spinor (WK-spinor) which is a generalization of a Killing spinor was proposed to construct a solution to the Einstein-Dirac system.
	A SqK-spinor was introduced as a candidate for a WK-spinor on a Sasakian manifold with constant scalar curvature.
	In particular, a WK-spinor is a SqK-spinor on a three-dimensional Sasakian space-form.
	A SqK-spinor is also regarded as a generalization of a Killing spinor on Sasakian manifolds.
	A Killing spinor is an important subject in both mathematics and physics and have been studied extensively \cite{Palomo-Lozano,Alekseevsky 1998,Leitner,Alekseevsky 2009,Moroianu,Friedrich 2003,Rugina}.
	A generalized Killing spinor is an interesting subject and has been studied, but much remains to be clarified.
	A SqK-spinor is a special case of the generalized Killing spinor, but there have been few studies.
	SqK-spinors are useful for constructing exact solutions to the Einstein-Dirac-Maxwell system \cite{Matsuno Ueno 2023}.
	It is worthwhile to investigate mathematical and physical properties of a SqK-spinor.
	
	The purpose of this paper is to investigate the geometric properties of a SqK-spinor and to show that a SqK-spinor solves to not only the Einstein-Dirac system but also the Einstein-Dirac-Maxwell system with a contact Maxwell field on Riemannian and Lorentzian Sasakian manifolds in three dimensions.
	
	The organization of this paper is as follows.
	In section \ref{sec:riemann Reviews and Preliminaries}, we briefly review some results on SqK-spinors and prepare some notations.
	In section \ref{sec:Mathematical and physical properties}, we investigate some geometric properties of a SqK-spinor.
	We give the conditions under which the Dirac current becomes a Killing vector field.
	We also clarify that the Dirac current of a certain SqK-spinor follows the equation of motion of a charged particle in the presence of a contact magnetic field, and a physical relationship between a Dirac-like current and a Maxwell field defined by SqK-spinors.
	In section \ref{sec:Solutions to DM, ED and EDM systems}, we completely determine weak Killing spinors in three-dimensional Lorentzian Sasakian manifolds and construct solutions to the Einstein-Dirac system similar to a discussion in \cite{Kim Friedrich 2000}.
	It is also shown that a SqK-spinor and a contact magnetic field solve to the Dirac-Maxwell and the Einstein-Dirac-Maxwell system in the three-dimensional pseudo-Riemannian Sasakian space-form.
	We also discuss energy conditons of SqK-spinors of all types.
	In section \ref{sec:Explicit formula of SqK-spinors}, explicit representations of SqK-spinors are given.
	In section \ref{sec:Conclusions}, we summarize our results and discuss its significance.
	
	\section{Preliminaries}\label{sec:riemann Reviews and Preliminaries}
	Here we prepare some basic formulae and review the existence theorem of a SqK-spinor in three dimensions.
	A Sasakian manifold is a contact metric manifold $(M,\phi,\xi,\eta,g)$ satisfying $R(X,Y)\xi=\eta(Y)X-\eta(X)Y$ \cite{blair}.
	For a vector $X$ orthogonal to $\xi$, a sectional curvature of the form $K(X,\phi(X))$ is called a $\phi$-sectional curvature, where $K(X,Y)$ denotes the sectional curvature of the plane section spanned by $X$ and $Y$.
	A simply connected Sasakian manifold whose $\phi$-sectional curvatures are all the same and constant is called a Sasakian space-form.
	
	Let $(M,\phi,\xi,\eta,g)$ be a Riemannian-Sasakian manifold.
	We define the Lorentzian metric $\tilde g=g-2\eta\otimes\eta$, then $(M,\phi,\xi,\eta,\tilde g)$ is a Lorentzian-Sasakian manifold\cite{Brunetti 2013}.
	Especially, Lorentzian-Sasakian space-form is obtained from Riemannian-Sasakian space-form by this correspondence.
	The relationship between a $\phi$-sectional curvature $K(X,\phi X)$ of $(M,\phi,\xi,\eta,g)$ and a $\phi$-sectional curvature $\widetilde K(X,\phi X)$ of $(M,\phi,\xi,\eta,\tilde g)$ is given by $\widetilde K(X,\phi X)=K(X,\phi X)+6$.
	
	Let $(M^{r,3-r},\phi,\xi,\eta,g)$ be a three-dimensional pseudo-Riemannian Sasakian manifold, where the index $r$ is zero for the Riemannian case or one for the Lorentzian case.
	We denote by $H$ the $\phi$-sectional curvature, then the Ricci tensor is given by
	\begin{align}
		{\rm Ric}= (H+(-1)^r)g +(1-(-1)^rH)\eta\otimes\eta.
	\end{align}
	
	Let $(M^{r,3-r},\phi,\xi,\eta,g)$ be a pseudo-Riemannian Sasakian manifold.
	A SqK-spinor of type $(a,b)$ is defined as a spinor satisfying
	\begin{align}
		\nabla_X\psi=a(\sqrt{-1})^rX\cdot\psi+b(\sqrt{-1})^r\eta(X)\xi\cdot\psi,
	\end{align}
	where $a$ and $b$ are real constants.
	
	From Lemma 6.4. of \cite{Kim Friedrich 2000}, if there exists a SqK-spinor of type $(a,b)$, then the scalar curvature $S$ satisfies
	\begin{align}
		S=(-1)^r(8n(2n+1)a^2+16nab). \label{eq:S,a,b}
	\end{align}
	A three-dimensional simply connected pseudo-Riemannian Sasakian manifolds with constant scalar curvature is a Sasakian space-form.
	Thus from now on we will consider a three-dimensional pseudo-Riemannian Sasakian space-form $M^{r,3-r}$ whose constant $\phi$-sectional curvature is $H$, and $M^{r,3-r}$ is classified as in the following Table \ref{tb:M and H}.
	
	\begin{table}[H]
		\caption{The relationship of $M$ and $\phi$-sectional curvature $H$}
  \label{tb:M and H}
		\centering
		\begin{tabular}{|c|c|c|c|}
			\hline
			$H$  & $H<(-1)^{r-1}3$ & $H=(-1)^{r-1}3$ & $(-1)^{r-1}3<H$ \\
			\hline
			$M^{r,3-r}$ & $\widetilde{SL(2,\mathbb{R})}$ & Nil$^{r,3-r}$ & $S^{r,3-r}$ \\
			\hline
		\end{tabular}
	\end{table}
	
	On a three-dimensional pseudo-Riemannian Sasakian space-form $(M^{r,3-r},\phi,\xi,\eta,g)$, there exists a global frame $\{e_1=\xi,e_2,e_3\}$ that satisfies
	\begin{align}
		[e_1,e_2]=-\frac{H+(-1)^r3}{2}e_3,\ [e_1,e_3]=\frac{H+(-1)^r3}{2}e_2,\ [e_2,e_3]=-2e_1.
	\end{align}
	We call this frame a Sasakian frame \cite{Matsuno Ueno 2023}.
	We denote the co-frame by $\{\omega^1=\eta,\omega^2,\omega^3\}$.
	The Riemannian connection with respect to a Sasakian frame is given by
	\begin{align}
		\nabla_{e_i}e_j=-\sum_{k=1}^3c_i\epsilon_{ij}^{~~k}e_k,\ 
		c_1=\frac{H+(-1)^r}{2},\ c_2=c_3=(-1)^r.
	\end{align}
	
	The Pauli matrices are given by
	\begin{align}
		\sigma_1=\begin{pmatrix}0 & 1 \\ 1 & 0\end{pmatrix},\ 
		\sigma_2=\begin{pmatrix}0 & -\sqrt{-1} \\ \sqrt{-1} & 0\end{pmatrix},\ 
		\sigma_3=\begin{pmatrix}1 & 0 \\ 0 & -1\end{pmatrix},
	\end{align}
	and we define the Gamma matrices $\gamma_1=(-\sqrt{-1})^{r-1}\sigma_1,\gamma_2=\sqrt{-1}\sigma_2,\gamma_3=\sqrt{-1}\sigma_3$.
	
	Then the spin connection form $\omega^S$ is given by
	\begin{align}
		\omega^S=\frac{(-\sqrt{-1})^r}{2}\sum_i\omega^i\gamma_i+(\sqrt{-1})^r\frac{H+(-1)^{r-1}}{4}\omega^1\gamma_1.\label{eq:spin conn form}
	\end{align}
	Now we define the SqK-connection $\nabla^{\rm SqK}$
	\begin{align}
		\nabla^{\rm SqK}_X\psi:=\nabla_X\psi-(\sqrt{-1})^r(aX\cdot\psi+b\eta(X)\xi\cdot\psi),
	\end{align}
	then a SqK-spinor of type $(a,b)$ is a $\nabla^{\rm SqK}$-parallel spinor.
	The curvature form $\Omega^{\rm SqK}$ is given by
	\begin{align}
		&\Omega^{\rm SqK}(e_i,e_j)=\sum_k\epsilon_{ijk}b_k\gamma_k,\\
		&b_1=(\sqrt{-1})^r\frac{H}{2} - 2(\sqrt{-1})^r a^{2} - 2 b,\\
		&b_2=b_3=-2(-\sqrt{-1})^ra^2-2(-\sqrt{-1})^rab+(-1)^rb+\frac{(\sqrt{-1})^r}{2}
	\end{align}
	Therefore, the integrability condition of a SqK-spinor is $\Omega^{\rm SqK}\equiv0$, that is $b_k=0$.
	Thus, we obtain the following theorem.
	
	\begin{proposition}
		Let $(M^3,\phi,\xi,\eta,g)$ be a three-dimensional Riemannian-Sasakian space-form with $\phi$-sectional curvature $H$.
		There exists a SqK-spinor of type $ \left(\frac{1}{2},\frac{H-1}{4}\right)$.
		In the case of $ H>-3$, $M$ is homeomorphic to $S^3$, then there exists a SqK-spinor of type $\left(\frac{1\pm \sqrt{3+H}}{2},\frac{-2\mp\sqrt{3+H}}{2}\right)$. (\cite{Kim Friedrich 2000},Theorem 8.6.)
	\end{proposition}
	
	\begin{proposition}
		Let $(M^{1,2},\phi,\xi,\eta,g)$ be a three-dimensional Lorentzian-Sasakian space-form with $\phi$-sectional curvature $H$.
		Then there exists a SqK-spinor of type $ \left(-\frac{1}{2},\frac{H+1}{4}\right)$.
		In the case of $H<3$, $M$ is homeomorphic to $\widetilde{SL(2,\mathbb{R})}$, then there exists a SqK-spinor of type $\left(\frac{-1 \pm \sqrt{- H+3}}{2}, \frac{2 \mp \sqrt{- H+3}}{2} \right)$.
	\end{proposition}
	
	We denote by $\mathcal{S}_0$ the set of SqK-spinors of type $ \left(\frac{(-1)^r}{2},\frac{H-(-1)^r}{4}\right)$ and by $\mathcal{S}_\pm$ the set of that of type $ \left(\frac{(-1)^r\pm \sqrt{3+(-1)^rH}}{2},\frac{(-1)^{r-1}2\mp\sqrt{3+(-1)^rH}}{2}\right)$, and we denote by $\Sigma$ the spinor bundle.
	
	\section{Geometric properties of SqK-spinors}\label{sec:Mathematical and physical properties}
	In this section, we discuss four geometric properties of SqK-spinors.
	First, the Reeb vector field $\xi$ defines the mapping which maps a SqK-spinor to another one.
	
	\begin{proposition}\label{pr:xi map}
		Let $(M^{r,3-r},\phi,\xi,\eta,g)$ be a three-dimensional pseudo-Riemannian Sasakian space-form with $\phi$-sectional curvature $H$, and let $\varphi$ be a SqK-spinor of type $(a,b)$.
		Then $\psi=\xi\varphi$ is a SqK-spinor of type $((-1)^r-a,(-1)^{r-1}+2a+b)$.
		Therefore the map $\xi:\Gamma(\Sigma)\ni\psi\mapsto\xi\psi\in\Gamma(\Sigma)$ is the map such that $\xi:\mathcal{S}_0\to\mathcal{S}_0$ and $\xi:\mathcal{S}_\pm\to\mathcal{S}_\mp$.
	\end{proposition}
	
	\begin{proof}
		The following calculation implies the claim.
		\begin{align*}
			\nabla_X\psi&=(\nabla_X\xi)\varphi+(\sqrt{-1})^ra\xi\cdot X\cdot\varphi+(\sqrt{-1})^rb\eta(X)\xi\cdot\psi\\
			&=(-\sqrt{-1})^r(X\cdot\xi\cdot+g(\xi,X))\varphi+a(\sqrt{-1})^r\xi\cdot X\cdot\varphi+b(\sqrt{-1})^r\eta(X)\xi\cdot\psi\\
			&=(-\sqrt{-1})^rX\cdot\psi-(-\sqrt{-1})^r\eta(X)\xi\cdot\psi-a(\sqrt{-1})^rX\cdot\psi\\
			&+2a(-\sqrt{-1})^rg(\xi,X)\xi\cdot\psi+b(\sqrt{-1})^r\eta(X)\xi\cdot\psi\\
			&=(\sqrt{-1})^r((-1)^r-a)X\cdot\psi+(\sqrt{-1})^r((-1)^{r-1}+2a+b)\eta(X)\xi\cdot\psi 
		\end{align*}
	\end{proof}
	
	On a three-dimensional pseudo-Riemannian spin manifold, a Spin-invariant bilinear form is given by $\langle\psi_1,\psi_2\rangle:=\psi_1^\dagger(\gamma_1)^r\psi_2$.
	The Dirac current of a spinor $\psi$ is defined as the vector field
	\begin{align}
		J_\psi:=\sum_{i=1}^3(-\sqrt{-1})^{r-1}\langle\psi, \gamma^i\psi\rangle e_i,
	\end{align}
	where $\{e_i\}_{i=1,2,3}$ is an orthonormal frame.
	For a SqK-spinor $\psi$, we have immediately
	$$
	X\langle\psi,\psi\rangle=\langle (\sqrt{-1})^raX\cdot\psi+(\sqrt{-1})^rb\xi\cdot\psi,\psi\rangle+\langle \psi, (\sqrt{-1})^raX\cdot\psi+(\sqrt{-1})^rb\xi\cdot\psi\rangle=0,
	$$
	so the quadratic form $\langle\psi,\psi\rangle$ is constant.
	A derivative formula of the Dirac current of a SqK-spinor is as follows.
	
	\begin{lemma}\label{lm:derivative of current}
		Let $(M^{r,3-r},\phi,\xi,\eta,g)$ be a three-dimensional pseudo-Riemannian Sasakian space-form with $\phi$-sectional curvautre $H$, and let $\{e_1,e_2,e_3\}$ be a Sasakian frame.
		Let $\psi$ be a SqK-spinor $\psi$ of type $(a,b)$, and put $J_\psi$ be the Dirac current of $\psi$ and $J_i:=g(J_\psi,e_i)$.
		Then we have
		\begin{align}
			&e_j(J_i) =\sum_k(2a\epsilon_{jik}+2b\eta(e_j) \epsilon_{1ik} -c_j\epsilon_{jik})J^k, \label{eq:derivative of current}\\
			&g(\nabla_{e_j}J_\psi,e_i)=\sum_k(2a\epsilon_{jik}+2b\eta(e_j) \epsilon_{1ik})J^k.\label{eq:nabla J}
		\end{align}
	\end{lemma}
	
	\begin{proof}
		By simple calculation, we have
		\begin{align}
			e_j(J_i)&=(-\sqrt{-1})^{r-1}\langle\nabla_{e_j}\psi,\gamma_i\psi\rangle +(-\sqrt{-1})^{r-1}\langle\psi,(\nabla_{e_j}\gamma_i)\psi\rangle +(-\sqrt{-1})^{r-1}\langle\psi,\gamma_i\nabla_{e_j}\psi\rangle\\
			&=(-\sqrt{-1})^{r-1}(-\sqrt{-1})^ra\langle \gamma_j\psi,\gamma_i\psi\rangle
			+a(-1)^{r-1}(\sqrt{-1})^{2r-1}\langle\psi,\gamma_i\gamma_j\psi\rangle\\
			&+ (-\sqrt{-1})^{r-1}(-\sqrt{-1})^rb\eta(e_j)\langle \gamma_1\psi,\gamma_i\psi\rangle
			+(-1)^{r-1}(\sqrt{-1})^{2r-1}b\eta(e_j)\langle\psi,\gamma_i\gamma_1\psi\rangle\\
			&+(-\sqrt{-1})^{r-1}\langle\psi,(\nabla_{e_j}e_i)\psi\rangle\\
			& =a\sqrt{-1}\langle\psi,-\gamma_j\gamma_i+\gamma_i\gamma_j\psi\rangle
			+b\sqrt{-1}\eta(e_j)\langle\psi,-\gamma_1\gamma_i+\gamma_i\gamma_1\psi\rangle\\
			&+(-\sqrt{-1})^{r-1}\langle\psi,(\nabla_{e_j}e_i)\psi\rangle,
		\end{align}
		and using $\gamma_i\gamma_j = -\eta_{ij}-(-1)^{r-1}(\sqrt{-1})^{r-2}\sum_k\epsilon_{ijk}\gamma^k {\rm\ and\ }\nabla_{e_i}e_j=-\sum_kc_i\epsilon_{ij}^{~~k}e_k$, we obtain
		\begin{align}
			e_j(J_i)&=2a(-\sqrt{-1})^{r-1}\epsilon_{jik}\langle\psi,\gamma^k\psi\rangle+2b\eta(e_j)(-\sqrt{-1})^{r-1}\epsilon_{1ik}\langle\psi,\gamma^k\psi\rangle\\
			&-(-\sqrt{-1})^{r-1}c_j\epsilon_{jik}\langle\psi,\gamma^k\psi\rangle\\
			&=\sum_k(2a\epsilon_{jik}+2b\eta(e_j) \epsilon_{1ik} -c_j\epsilon_{jik})J^k
		\end{align}
	\end{proof}
	
	Next, we clarify the necessary and sufficient condition for the Dirac current of a SqK-spinor to be a Killing vector.
	\begin{proposition}\label{pr:killing vector}
		Let $(M^{r,3-r},\phi,\xi,\eta,g)$ be a three-dimensional pseudo-Riemannian Sasakian space-form with $\phi$-sectional curvautre $H$, and let $\psi$ be a SqK-spinor but not a Killing spinor.
		Then $J_\psi$ is a Killing vector field if and only if $\psi$ satisfies $\xi\cdot\psi=\pm(-\sqrt{-1})^{r-1}\psi$.
		Furthermore, in this case, we have $\psi\in\mathcal{S}_0$ and $J_\psi$ is proportional to the Reeb vector field $\xi$.
	\end{proposition}
	\begin{proof}
		From \eqref{eq:nabla J}, we have
		\begin{align}
			g(\nabla_{e_1}J_\psi,e_2)+g(\nabla_{e_2}J_\psi,e_1)&=2bJ_3,\\
			g(\nabla_{e_1}J_\psi,e_3)+g(\nabla_{e_3}J_\psi,e_1)&=-2bJ_2,\\
			g(\nabla_{e_2}J_\psi,e_3)+g(\nabla_{e_3}J_\psi,e_2)&=0,
		\end{align}
		thus $J_\psi$ is a Killing vector field if and only if $J_2=J_3=0$ holds, and then it is proportional to $\xi$.
		It is obvious that $\psi$ is an eigen spinor of $\gamma_1$ if and only if $J_2=J_3=0$, and then the eigen values are $\pm(-\sqrt{-1})^{r-1}$.
		
		For a SqK-spinor of type $(a,b)$, from \eqref{eq:derivative of current}, we have $e_2(J_3) =(2a-(-1)^r)J^1$.
		If $J_\psi$ is a Killing vector field, then $J_1\ne0,J_2=J_3=0$, which implies $a=(-1)^r/2$, therefore $\psi\in\mathcal{S}_0$.
	\end{proof}
	
	\begin{remark}
		A SqK-spinor in $\mathcal{S}_0$ satisfies $\xi\cdot\psi=\pm(-\sqrt{-1})^{r-1}\psi$ is an example of a spinor that is not a Killing spinor and its Dirac current is a Killing vector field unless $M^{r,3-r}$ is a constant curvature space.
	\end{remark}
	
	Next, we show that the Dirac current of a SqK-spinor follows the equation of motion of a charged particle in the presence of a contact Maxwell field.
	A Maxwell field is given by a closed two-form.
	A charged particle with charge $q$ in the presence of a Maxwell field $F$ is governed by the following equation
	\begin{align}
		\nabla_{\dot c}\dot c=q\ {}^\sharp(\iota_{\dot c}F),
	\end{align}
	where $c:I\to M$ is an orbit of the charged particle.
	
	In \cite{Cabrerizo 2009}, on a contact metric manifold $(M,\phi,\xi,\eta,g)$, a contact magnetic field is defined as a magnetic field $F=d\eta$, that is, a representative of the gauge field is $\eta$.
	We call $F=d\eta$ a contact Maxwell field on a Riemannian or a Lorentzian contact metric manifold.
	
	\begin{proposition}\label{pr:magnetic curve}
		Let $(M^{r,3-r},\phi,\xi,\eta,g)$ be a three-dimensional pseudo-Riemannian Sasakian space-form with the $\phi$-sectional curvautre $H$, and $\psi$ be a SqK-spinor of type $(a,b),\ b\ne0$.
		Then $g(J_\psi,\xi)$ is constant along each integral curve of $J_\psi$.
		Furthermore,
		\begin{align}
			\nabla_{J_\psi}J_\psi=(-1)^{r-1}bg(J_\psi,\xi){}^\sharp(\iota_{J_\psi}d\eta),
		\end{align}
		holds, that is, each integral curve of $J_\psi$ is an orbit of a charged particle with charge $(-1)^{r-1}bg(J_\psi,\xi)$ in the presence of the contact Maxwell field $d\eta$ unless $g(J_\psi,\xi)\ne0$.
	\end{proposition}
	
	\begin{proof}
		Lemma \ref{lm:derivative of current} implies $ e_1(J_1)=0,\ e_2(J_1)=(-2a+(-1)^r)J_3,\ e_3(J_1)=(2a-(-1)^r)J_2$, and thus we have $ J_\psi(J_1)=J_2e_2(J_1)+J_3e_3(J_1)=0$.
		Then $ J_1$ is constant along each integral curve of $ J_\psi$.
		
		From \eqref{eq:nabla J}, we have $g(\nabla_{J_\psi}J_\psi,e_i)=2b\eta(J_\psi)\sum_k\epsilon_{1ik}J^k$, then
		\begin{align}
			g(\nabla_{J_\psi}J_\psi,e_1)=0,\ g(\nabla_{J_\psi}J_\psi,e_2)=(-1)^r2bJ_1J_3,\ g(\nabla_{J_\psi}J_\psi,e_3)=-(-1)^r2bJ_1J_2,
		\end{align}
		holds.
		Since we have $d\eta=2\omega^2\wedge\omega^3$, thus $\iota_{J_\psi}d\eta=2J_2\omega^3-2J_3\omega^2$ holds, it follows that
		\begin{align}
			\nabla_{J_\psi}J_\psi=(-1)^{r-1}bJ_1{}^\sharp(\iota_{J_\psi}d\eta).
		\end{align}
	\end{proof}
	
	\begin{corollary}
		In the same settings of Proposition \ref{pr:magnetic curve}, if $J_\psi$ is orthogonal to $\xi$ at some point, then it always be, and in this case the integral curve of $J_\psi$ is a geodesic.
	\end{corollary}
	
	Finally, we clarify a physical relation between a Dirac-like current and a Maxwell field defined by SqK-spinors.
	An electric current vector field $J$ generates a Maxwell field $F$ according to the Maxwell equation
	\begin{align}
		*d* F=(-1)^{r-1}\ {}^\flat J.
	\end{align}
	We can construct the solutions to the Maxwell equation using SqK-spinors.

	\begin{proposition}\label{pr:maxwell with source}
		Let $(M^{r,3-r},\phi,\xi,\eta,g)$ be a three-dimensional pseudo-Riemannian Sasakian space-form with the $\phi$-sectional curvautre $H$, and let $\psi_1$ and $\psi_2$ be SqK-spinors.
		We define a complex two-form $F_{\psi_1, \psi_2}$ as
		\begin{align}
			F_{\psi_1,\psi_2}=\frac{1}{2}\sum_{i,j}\sqrt{-1}\langle\psi_1,[\gamma_i,\gamma_j]\psi_2\rangle\omega^i\wedge\omega^j,
		\end{align}
		and define a Dirac-like complex current $J_{\psi_1,\psi_2}=\sum_i(-\sqrt{-1})^{r-1}\langle\psi_1, \gamma^i\psi_2\rangle e_i$.
		Then $dF_{\psi_1,\psi_2}=0$ holds if and only if 
		the triple $(H,\psi_1,\psi_2)$ is one of the cases in the following table\ref{tb:F closed}.

\begin{table}[H]
			\caption{The value of $H$ and $r$ for $F_{\psi_1,\psi_2}$ to be closed. }
   \label{tb:F closed}
			\centering
\begin{tabular}{|c|c|c|c|}
\hline
 & $\psi_1\in\mathcal{S}_0$ & $\psi_1\in\mathcal{S}_+$ & $\psi_1\in\mathcal{S}_-$  \\
\hline
$\psi_2\in\mathcal{S}_0$ & Any $H\in\mathbb{R},\ r\in\{0,1\}$ & $H=13,r=0$ & $H=-13,r=1$ \\
\hline
$\psi_2\in\mathcal{S}_+$ & $H=13, r=0$ & $(-1)^{r+1}H<3$ & None \\
\hline
$\psi_2\in\mathcal{S}_-$ & $H=-13, r=1$ & None &  $(-1)^{r+1}H<3$\\
\hline				
\end{tabular}
\end{table}

		And then, $F_{\psi_1,\psi_2}$ and $J_{\psi_1,\psi_2}$ satisfy the equation $*d*F_{\psi_1,\psi_2}=(-1)^{r-1}c\ {}^\flat J_{\psi_1,\psi_2}$ for some $c\in\mathbb{R}$ if and only if one of the following case holds:\\
		(i) If $\psi_1,\psi_2\in\mathcal{S}_0$ and $\psi_1$ is proportional to $\psi_2$, and both are eigen-spinors of $\xi$, where $H$ is arbitrary, then $c=4$ holds.\\
		(ii) If $H=(-1)^r,\ \psi_1,\psi_2\in\mathcal{S}_0$ holds, that is Killing spinors, then $c=4$ holds.\\
		(iii) If $H=13,r=0$ and $\psi_1\in S_0,\psi_2\in S_+$ or $\psi_1\in S_+,\psi_2\in S_0$ holds, then $c=12$ holds.\\
		(iv) If $H=-13,r=1$ and $\psi_1\in S_0,\psi_2\in S_-$ or $\psi_1\in S_-,\psi_2\in S_0$ holds, then $c=12$ holds.\\
		(v) If $ \psi_1,\psi_2\in \mathcal{S}_0$ and   $\langle\psi_1,\gamma_1\psi_2\rangle=0$ holds, where $H$ is arbitrary, then $c=3+(-1)^rH$ holds.\\
		
	\end{proposition}

	\begin{proof}
 Let $\psi_1$ and $\psi_2$ be SqK-spinors of type $(a_1,b_1)$ and $(a_2,b_2)$ respectively.
		Since $[\gamma_i,\gamma_j]= -2(-\sqrt{-1})^r\epsilon_{ij}^{~~k}\gamma_k$ holds, we have
		\begin{align}
			F= (-\sqrt{-1})^{r+1}\epsilon_{ij}^{~~k}\langle \psi_1, \gamma_k\psi_2 \rangle \omega^{i}\wedge\omega^j.
		\end{align}
		Then $dF$ is given by
		\begin{align}
			dF &= (-\sqrt{-1})^{r+1}\epsilon_{ij}^{~~k}\{ \langle \nabla_l\psi_1, \gamma_k\psi_2 \rangle + \langle \psi_1, \nabla_l(\gamma_k\psi_2) \rangle \} \omega^l\wedge\omega^{i}\wedge\omega^j\\
			&= (-\sqrt{-1})^{r+1}\epsilon_{ij}^{~~k}\epsilon^{lij}\{ \langle \nabla_l\psi_1, \gamma_k\psi_2 \rangle + \langle \psi_1, \nabla_l(\gamma_k\psi_2) \rangle \} \omega^1\wedge\omega^2\wedge\omega^3\\
			&= 2(-\sqrt{-1})^{r+1}\eta^{kl}\{ \langle \nabla_l\psi_1, \gamma_k\psi_2 \rangle + \langle \psi_1, \nabla_l(\gamma_k\psi_2) \rangle \} \omega^1\wedge\omega^2\wedge\omega^3\\
			&= 2(-\sqrt{-1})^{r+1}\eta^{kl}( \langle (\sqrt{-1})^r(a_1\gamma_l+b_1\delta_{1l}\gamma_1)\psi_1, \gamma_k\psi_2 \rangle\\
			&+ \langle \psi_1, \gamma_k(\sqrt{-1})^r(a_2\gamma_l+b_2\delta_{1l}\gamma_1)\psi_2) \rangle )\omega^1\wedge\omega^2\wedge\omega^3\\
			&= 2(-1)^{r+1}\sqrt{-1}( a_1((-1)^r\langle \gamma_1\psi_1, \gamma_1\psi_2 \rangle + \langle \gamma_2\psi_1, \gamma_2\psi_2 \rangle \\
			&+ \langle \gamma_3\psi_1, \gamma_3\psi_2 \rangle) + (-1)^rb_1\langle \gamma_1\psi_1, \gamma_1\psi_2 \rangle ) \omega^1\wedge\omega^2\wedge\omega^3\\
			& -2\sqrt{-1}( a_2((-1)^r\langle \psi_1, (\gamma_1)^2\psi_2 \rangle + \langle \psi_1, (\gamma_2)^2\psi_2 \rangle \\
			&+ \langle \psi_1, (\gamma_3)^2\psi_2 \rangle) + (-1)^rb_2\langle \psi_1, (\gamma_1)^2\psi_2 \rangle )\omega^1\wedge\omega^2\wedge\omega^3\\
			&= 2\sqrt{-1}\{ -3a_1 +3a_2 - b_1 + b_2\}\langle\psi_1,\psi_2 \rangle\omega^1\wedge\omega^2\wedge\omega^3.
		\end{align}

		Therefore $dF=0$ holds if and only if the triple $(H,\psi_1,\psi_2)$ is in the table \ref{tb:F closed}.
		Furthermore we have
		\begin{align}
			\nabla_iF_{jk}&=\sqrt{-1}\langle\nabla_i\psi_1,[\gamma_j,\gamma_k]\psi_2\rangle+\sqrt{-1}\langle\psi_1,[\gamma_j,\gamma_k]\nabla_i\psi_2\rangle \\
			&=(\sqrt{-1})^{r-1}\langle a_1\gamma_i\psi_1+b_1\delta_{i1}\gamma_1\psi_1,[\gamma_j,\gamma_k]\psi_2\rangle\\
			&-(\sqrt{-1})^{r-1}\langle\psi_1,[\gamma_j,\gamma_k](a_2\gamma_i\psi_2+b_2\delta_{i1}\gamma_1\psi_2)\rangle \\
			&=a_1(\sqrt{-1})^{r-1}\langle \gamma_i\psi_1,[\gamma_j,\gamma_k]\psi_2\rangle+\delta_{i1}b_1(\sqrt{-1})^{r-1}\langle \gamma_1\psi_1,[\gamma_j,\gamma_k]\psi_2\rangle\\
			&-a_2(\sqrt{-1})^{r-1}\langle\psi_1,[\gamma_j,\gamma_k]\gamma_i\psi_2\rangle-b_2\delta_{i1}(\sqrt{-1})^{r-1}\langle\psi_1,[\gamma_j,\gamma_k]\gamma_1\psi_2\rangle.
		\end{align}

		Using $ \sum_i(\gamma^i\gamma_i\gamma_k-\gamma^i\gamma_k\gamma_i)=-4\gamma_k,\  \gamma_1\gamma^1\gamma_k-\gamma_1\gamma_k\gamma^1=-2\gamma_k+2\delta_{1k}\gamma_1$, we obtain
		\begin{align}
			\sum_i\nabla^iF_{ik}&=(a_1+a_2)(\sqrt{-1})^{r-1}\langle \psi_1,\gamma^i[\gamma_i,\gamma_k]\psi_2\rangle+(b_1+b_2)(\sqrt{-1})^{r-1}\langle \psi_1,\gamma_1[\gamma^1,\gamma_k]\psi_2\rangle \\
			&=-(4(a_1+a_2)+2(b_1+b_2))(\sqrt{-1})^{r-1}\langle\psi_1,\gamma_k\psi_2\rangle+2(b_1+b_2)\delta_{1k}(\sqrt{-1})^{r-1}\langle\psi_1,\gamma_1\psi_2\rangle\\
			&=(-1)^r(4(a_1+a_2)+2(b_1+b_2))J_k+(-1)^{r-1}2(b_1+b_2)\delta_{1k}J_1.
		\end{align}
		Therefore we obtain
		\begin{align}
			*d*F&=(-1)^{r-1}\nabla^iF_{ik}\omega^k\\
			&=-(4(a_1+a_2)+2(b_1+b_2)){}^\flat J_{\psi_1,\psi_2}+2(b_1+b_2)g(J_{\psi_1,\psi_2},\xi)\omega^1.
		\end{align}
		
		We consider the conditions in which $*d*F_{\psi_1,\psi_2}=(-1)^{r-1}c\ {}^\flat J_{\psi_1,\psi_2}$ holds for some $c\in\mathbb{R}$.
		We can assume $J_{\psi_1,\psi_2}\ne0$ since $\psi_1,\psi_2$ are not zero-spinor.
		For $ \psi_1\in \mathcal{S}_0,\psi_2\in \mathcal{S}_\pm$, the quadratic $ \langle\psi_1,\gamma_i\psi_2\rangle$ is not identically zero because of an explicit formula (see section \ref{sec:Explicit formula of SqK-spinors}), then it is not possible that $ J_{\psi_1,\psi_2}$ is proportional to $ \xi$ or orthogonal to $\xi$ on the whole of $M$.
		As is the same for the case of $ \psi_1\in \mathcal{S}_\pm,\psi_2\in \mathcal{S}_0$ or the case of $\psi_1, \psi_2\in\mathcal{S}_\pm$.
		
		If $ J_{\psi_1,\psi_2}$ is proportional to $ \xi$, we have $ \psi_1, \psi_2\in S_0$ and $ \langle\psi_1,\gamma_2\psi_2\rangle=\langle\psi_1,\gamma_3\psi_2\rangle=0$, which is equivalent to that the both $\psi_1$ and $ \psi_2 $ are proportional to the eigen-spinors of $\xi$, this is the case (i), and then we have
		\begin{align}
			*d*F_{\psi_1,\psi_2}=4(-1)^{r-1}\ {}^\flat J_{\psi_1,\psi_2}.\label{eq:maxwell with source (i)}
		\end{align}
		
		If $ J_{\psi_1,\psi_2}$ is not proportional to $ \xi$, we have $(b_1+b_2)g(J_{\psi_1,\psi_2},\xi)=0$.
		If $b_1+b_2=0$ holds, it is possible for the cases (ii) $H=(-1)^r$, (iii) $r=0, H=13, \psi_1\in \mathcal{S}_0,\psi_2\in \mathcal{S}_+$ or $ \psi_1\in \mathcal{S}_+,\psi_2\in \mathcal{S}_0$, or (iv) $r=1,H=-13, \psi_1\in \mathcal{S}_0,\psi_2\in \mathcal{S}_-$ or $ \psi_1\in \mathcal{S}_-,\psi_2\in \mathcal{S}_0$.
		In the case (ii), we have \eqref{eq:maxwell with source (i)}, and in the cases (iii),(iv), we have
		\begin{align}
			*d*F_{\psi_1,\psi_2}=(-1)^{r-1}12\ {}^\flat J_{\psi_1,\psi_2}.
		\end{align}
		
		If $b_1+b_2\ne0$ holds, then we have $ g(J_{\psi_1,\psi_2},\xi)=0$, which implies (v) $ \psi_1,\psi_2\in \mathcal{S}_0$ and $ \langle\psi_1,\gamma_1\psi_2\rangle \equiv 0$, and then we have
		\begin{align}
			*d*F_{\psi_1,\psi_2}=(-1)^{r-1}(3+(-1)^rH)\ {}^\flat J_{\psi_1,\psi_2}.
		\end{align}
		
	\end{proof}
	
	\begin{remark}
		According to the above proposition, $\Re (F_{\psi_1,\psi_2})$ and $\Re (J_{\psi_1,\psi_2})$ constitute a solution to the Maxwell system.
		In particular, on the Riemannian or Lorentzian Nill manifold, the two-form $\Re (F_{\psi_1,\psi_2})$ is a vacuum Maxwell field.
	\end{remark}
	
	\section{Solutions to DM, ED and EDM systems}\label{sec:Solutions to DM, ED and EDM systems}
	In this section, we discuss properties of SqK-spinors from the viewpoint of general relativity.
	We construct solutions to Einstein-Dirac (ED) system and Einstein-Dirac-Maxwell (EDM) system using SqK-spinors.
	
	On an $n$-dimensional pseudo-Riemannian spin-c manifold $(M^{r,n-r},g)$, let $\psi$ be a spinor and $\sqrt{-1}A$ be a $U(1)$-gauge field.
	We denote by $\nabla^c:=\nabla+q\sqrt{-1}A$ a spin-c connection, and the Dirac operator is defined as 
	\begin{align}
	D^c=(\sqrt{-1})^r\gamma^i\nabla^c_{e_i},    
	\end{align}
	where $\{e_i\}$ is an orthogonal frame.
	For a triple $(\psi,A,g)$, the Lagrangian density of an EDM system is given by 
	\begin{align}
&\mathcal{L}_{EDM}=\mathcal{L}_{Ein}+\mathcal{L}_{Dirac}+2\mathcal{L}_{EM}\label{eq:Lagrangian E-D-M}\\
&\mathcal{L}_{Ein}=S-2\Lambda,\ 
\mathcal{L}_{Dirac}=-\Re\langle\psi,(D^c-\lambda)\psi\rangle,\ 
\mathcal{L}_{EM}=-\frac{1}{4}||F||^2
	\end{align}
	where $S$ is the scalar curvature of $M$, $F=dA$ is the Maxwell field, $\Lambda$ is a cosmological constant and $\lambda$ is a real constant.
	
	Then the equations of the fields are given by
	\begin{align}
		&D^c\psi=\lambda\psi,\label{eq:dirac}\\
		&*d*F=q(-1)^{n+1}\ {}^\flat J_\psi,\label{eq:maxell}\\
		&J_\psi=(-\sqrt{-1})^{r-1}\sum_i\langle\psi,\gamma^i\psi\rangle e_i,\\
		&{\rm Ric}=T-\frac{1}{n-2}\tr(T)g+\frac{2}{n-2}\Lambda g,\label{eq:einstein}\\
		&T=T^{em}+T^{spin},\\
		&T^{em}(X,Y)=g^*(\iota_XF,\iota_YF)-\frac{1}{4}||F||^2g(X,Y),\\
		&T^{spin}(X,Y)=\frac{1}{4}\Re\langle\psi,(\sqrt{-1})^r(X\nabla^c_Y\psi+Y\nabla^c_X\psi)\rangle.
	\end{align}
	The Lagrangian \eqref{eq:Lagrangian E-D-M} with $A=0$ describes a ED system, and then equations of the fields are \eqref{eq:dirac} and \eqref{eq:einstein}.
	The Lagrangian with terms of scalar curvature $S$ and cosmological constant $-2\Lambda$ removed from Lagrangian\eqref{eq:Lagrangian E-D-M} describes a Dirac-Maxwell (DM) system, and then equations of the fields are \eqref{eq:dirac} and \eqref{eq:maxell}.
	
	In \cite{Kim Friedrich 2000},\cite{Kim 2006}, a weak Killing spinor is defined as the following.
	
	\begin{definition}\label{def:wk}
		Let $(M^{r,n-r},g)$ be an n-dimensional pseudo-Riemannian spin manifold.
		Let $\Lambda$ be a real number and assume that $(n-2)S - 2n\Lambda$ does not vanish at any point of $ (M,g)$ for $ n \ge 3$.
		A spinor $ \psi\ne0$ is called a weak Killing spinor (WK-spinor) with WK-number $(\sqrt{-1})^{3r}\lambda$ where, $ 0\neq \lambda \in \mathbb{R} $ if $ \psi $ satisfies the following
		\begin{align}
			\nabla_X\psi = (\sqrt{-1})^{3r}\beta(X)\cdot\psi+n\alpha(X)\psi+X\cdot\alpha\cdot\psi,
		\end{align}
		where $ \alpha$ is a one-form and $ \beta$ is a symmetric (1,1)-tensor field defind by
		\begin{align}
			\alpha = \frac{(n-2)dS}{2(n-1)\{(n-2)S-2n\Lambda\}},\ 
			\beta = \frac{2\lambda}{(n-2)S - 2n\Lambda}\left\{Ric-\frac{1}{2}Sg+\Lambda g\right\}.
		\end{align}
	\end{definition}
	
	They proved the theorem which states the relationship between the ED system and WK-spinor.
	We summarize the result briefly.
	The key lemma is as follows.
	
	\begin{lemma}\label{lm:WK-spninor}
		Let $\psi$ be a non-null spinor field on an n-dimensional pseudo-Riemannian spin manifold $(M^{r,n-r},g)$ such that 
		\begin{equation}
			\nabla_X\psi=n\alpha(X)\psi+(\sqrt{-1})^{3r}\beta(X)\cdot\psi +X\cdot\alpha\cdot\psi \label{eq:wk-spinor 1}
		\end{equation}
		holds for a one-form $\alpha$ and a symmetric $(1,1)$-tensor field $\beta$ for all vector fields $X$.
		Then $\alpha$ and $\beta$ are uniquely determined by the spinor field $\psi$ via the relations
		\begin{equation}
			\alpha  =\frac{d\langle\psi,\psi\rangle}{2(n-1)\langle\psi,\psi\rangle},\quad
			\beta = -\frac{2T^{spin}}{\langle\psi,\psi\rangle}.    \label{eq:wk-spinor 2}
		\end{equation}
		Furthermore, in the case of $n=3$, any non-null spinor satisfies \eqref{eq:wk-spinor 1},\eqref{eq:wk-spinor 2}.
	\end{lemma}
	
	Then we have the following theorem.
	
	\begin{theorem}\label{thm:wk and ED}
		Let $\varphi$ be a non-null WK-spinor on an n-dimensional pseudo-Riemannian spin manifold $(M^{r,n-r},g).$
		Then $\frac{\langle\varphi,\varphi\rangle}{(n-2)S-2n\Lambda}$ is constant on $M$ and if the signs of $-\frac{(n-2)S-2n\Lambda}{\lambda}$ and $\langle\varphi,\varphi\rangle$ are the same, the pair of a pseudo-Riemannian metric and a WK-spinor $\left(g,\psi\coloneqq \sqrt{-\frac{(n-2)S-2n\Lambda}{\lambda\langle\varphi,\varphi\rangle}}\varphi\right)$ is a solution to the ED system (Theorem 3.3. in \cite{Kim Friedrich 2000}).
		Furthermore, in the case of $n=3$, the converse holds (Theorem 3.6. in \cite{Kim Friedrich 2000}).
	\end{theorem}
	
	\begin{proof}
		\begin{align}
			d\left(\frac{\langle\varphi,\varphi\rangle}{(n-2)S-2n\Lambda}\right)  = \frac{\{(n-2)S-2n\Lambda\}d\langle\varphi,\varphi\rangle-\langle\varphi,\varphi\rangle(n-2)dS}{\{(n-2)S-2n\Lambda\}^2}
			= 0.
		\end{align}
		Since $\langle\psi,\psi\rangle = -\frac{(n-2)S-2n\Lambda}{\lambda\langle\varphi,\varphi\rangle}\langle\varphi,\varphi\rangle = -\frac{(n-2)S-2n\Lambda}{\lambda}, $
		then, since the Lemma \ref{lm:WK-spninor}, 
		\begin{align}
			\beta  = -\frac{2T^{spin}}{\langle\psi,\psi\rangle} 
			= \frac{2\lambda}{(n-2)S-2n\Lambda}T^{spin}.
		\end{align}
		It implies that $T^{spin} = \operatorname{Ric}-\frac{1}{2}Sg+\Lambda g.$
	\end{proof}
	
	We obtain immediately by definition a relationship between a WK-spinor and a SqK-spinor.
	\begin{proposition}
		On a three-dimensional pseudo-Riemannian Sasakian space-form, a WK-spinor is a SqK-spinor.
	\end{proposition}
	\begin{proof}
		The Ricci tensor is given by $ {\rm Ric}= (H+(-1)^r)g +(1-(-1)^rH)\eta\otimes\eta $, which implies the claim.
	\end{proof}
	
	In Theorem 8.6. of \cite{Kim Friedrich 2000}, on a three-dimensional Sasakian space-form, the SqK-spinors which are WK-spinors with $\Lambda=0$ are decided.
	By almost the same argument we obtain the following results.
	\begin{proposition}\label{pr:wk riemann}
		Let $(M^{r,3-r},\phi,\xi,\eta,g)$ be a three-dimensional pseudo-Riemannian Sasakian space-form with the $\phi$-sectional curvautre $H$.
		A SqK-spinor of type $(a,b)$ which is a WK-spinor with WK-number $\lambda$ is the one of the following.
		
		(i) $ (a,b) =\left(\frac{(-1)^r}{2}, \frac{H-(-1)^r}{4}\right)$ case,
		$$
		\lambda=\frac{(-1)^{r-1}H-5}{4},\ \Lambda=(-1)^{r-1}.
		$$
		
		(ii) $(a,b) =\left(\frac{(-1)^r}{2} + \frac{\sqrt{(-1)^r H+3}}{2}, -\frac{\sqrt{(-1)^r H+3}}{2}+(-1)^{r-1}\right)$ and $(-1)^rH>-3$ case,
		$$
		\lambda=(-1)^{r-1}\sqrt{(-1)^rH+3}-\frac{1}{2},\ \Lambda=-\sqrt{3+(-1)^rH} + H +(-1)^r2 .
		$$
		
		(iii) $ (a,b) =\left(\frac{(-1)^r}{2} - \frac{\sqrt{(-1)^r H+3}}{2}, \frac{\sqrt{(-1)^r H+3}}{2}+(-1)^{r-1}\right)$ and $(-1)^rH>-3$ case,
		$$
		\lambda=(-1)^r\sqrt{(-1)^rH+3}-\frac{1}{2},\ \Lambda=  \sqrt{3+(-1)^rH} + H +(-1)^r2.
		$$
		
	\end{proposition}
	
	\begin{proof}
		Since we have $ {\rm Ric}= (H+(-1)^r)g +(1-(-1)^rH)\eta\otimes\eta $, then a WK-spinor $\psi$ satisfies
		\begin{align}
			&\nabla_X\psi =(\sqrt{-1})^{3r} \beta(X)\cdot\psi, \\
			&\beta = \frac{\lambda}{H-3\Lambda+(-1)^r2}\left\{\left( \Lambda -(-1)^r\right)g -(H-(-1)^r)\eta\otimes\eta \right\}.
		\end{align}
		Then a SqK-spinor $ \psi$ of type $ (a,b)$ is a WK-spinor if and only if
		\begin{align}
			(a,b) = \left(\frac{(-1)^r\lambda\left(\Lambda-(-1)^r\right)}{H-3\Lambda+(-1)^r2}, -\frac{(-1)^r\lambda(H-(-1)^r)}{H-3\Lambda+(-1)^r2} \right).
		\end{align}
		For the three cases, solving the above equations gives the claim.
	\end{proof}
	
	Therefore we can construct solutions to the ED system.
	\begin{proposition}\label{pr:solutions to ED}
		Let $(M^{r,3-r},\phi,\xi,\eta,g)$ be a three-dimensional pseudo-Riemannian Sasakian space-form with the $\phi$-sectional curvautre $H$.
		Let $\psi$ be a non-null SqK-spinor of one of the type (i),(ii) and (iii) in Proposition \ref{pr:wk riemann}.
		We assume that $H<1$ only in Riemannian case (ii).
		Then a pair $\left(g,\sqrt{-\frac{S-6\Lambda}{\lambda\langle\psi,\psi\rangle}}\psi\right)$ is a solution to the ED system with the corresponding eigenvalue $\lambda$ and cosmological constant $\Lambda$ in each case.
	\end{proposition}
	
	\begin{proof}
		According to Proposition \ref{thm:wk and ED}, it is enough to verify that the signs of $-\frac{S-6\Lambda}{\lambda}$ and $\langle\psi,\psi\rangle$ are the same for each case (i),(ii) and (iii) in Proposition \ref{pr:wk riemann}.
		In the Riemannian case, we have $-\frac{S-6\Lambda}{\lambda}=8$ for the case (i), and we have $-\frac{S-6\Lambda}{\lambda}=\frac{4 \left(H - 1\right) \left(2 \sqrt{H + 3} - 1\right)}{2 H - 5 \sqrt{H + 3} + 8}>0$ for the case (iii), and we have $-\frac{S-6\Lambda}{\lambda}=- \frac{4 \left(H - 1\right) \left(2 \sqrt{H + 3} + 1\right)}{2 H + 5 \sqrt{H + 3} + 8}>0$ if $-3<H<1$ for the case (ii).
		Since $\langle\psi,\psi\rangle$ is always positive, we obtain our Riemannian claim.
		In the Lorentzian case, we can choose the sign of $\langle\psi,\psi\rangle$ so that it coincide with that of $-\frac{S-6\Lambda}{\lambda}$ for each case (i),(ii) and (iii) (see Proposition \ref{pr:explicit formula}).
	\end{proof}
	
	We would like to know physical properties of SqK-spinors of each type, so we investigate the energy conditions, that is the null energy condition (NEC), the weak energy condition (WEC), the dominant energy condition (DEC) and the strong energy condition (SEC) for above Lorentzian solutions.
	For an energy-momentum tensor $T$, the NEC is $T(X,X)\geq0$ for any null $X$, the WEC is $T(X,X)\geq0$ for any timelike $X$, the DEC is that $v=-{}^\sharp T(X,\cdot)$ is future-pointing causal for any future-pointing causal $X$, and SEC is $(T-\frac{1}{n-2}\tr(T)g)(X,X)\geq0$ for any timelike $X$.
 
	\begin{proposition}\label{pr:energy condition}
		Let $(M^{1,2},\phi,\xi,\eta,g)$ be a three-dimensional Lorentzian Sasakian space-form with the $\phi$-sectional curvautre $H$ and $\psi$ be a SqK-spinor, and suppose that $(g,\psi)$ is one of the Lorentzian solutions in Proposition \ref{pr:solutions to ED}.
		Then the ranges of $H$ for which each energy condition of the spinor is satisfied are given in the following Table \ref{tb:range H ED}.
		\begin{table}[H]
			\caption{The ranges of $H$ for which  each energy condition is satisfied}
   \label{tb:range H ED}
			\centering
			\begin{tabular}{|c|c|c|c|c|}
				\hline
				Type & NEC & WEC & DEC & SEC \\
				\hline\hline
				(i)  & $-1\le H$ & $1\le H$ & $3\le H$ & $-1\le H$ \\
				\hline
				(ii)  & $-1\le H$ & $-1\le H<3$ & $-1\le H<3$ & $2\le H<3$ \\
				\hline
				(iii)  & $-1\le H$ & $-1\le H<3$ & None & $-1\le H<3$ \\
				\hline
			\end{tabular}
		\end{table}
	\end{proposition}
	\begin{proof}
		Since we have $T^{spin}={\rm Ric}-\frac{1}{2}Sg+\Lambda g=(1+\Lambda)g+(1+H)\eta\otimes \eta$, for a tangent vector $X=te_1+xe_2+ye_3$, we obtain
		\begin{align}
			&T^{spin}(X,X)=(H-\Lambda)t^2+(1+\Lambda)(x^2+y^2),\\
			&v=-{}^\sharp T^{spin}(X,\cdot)=(H-\Lambda)te_1-(1+\Lambda)(xe_2+ye_3),\\
			&||v||^2=-t^2(H-\Lambda)^2+(1+\Lambda)^2(x^2+y^2),\\
			&(T^{spin}-\tr(T)g)(X,X)=2(1+\Lambda)t^2+(H-1-2\Lambda)(x^2+y^2).
		\end{align}
		For any null $X$, the NEC $T^{spin}(X,X)=(H+1)t^2\geq0$ holds iff $-1\le H$ holds.
		For any timelike $X$, the WEC $ T^{spin}(X,X)\geq0$ holds iff $ H-\Lambda\geq0,\  1+\Lambda\geq0$ or $ H-\Lambda\geq0,\  1+\Lambda<0,\ H-\Lambda>|1+\Lambda|$ holds.
		For any future-pointing causal $X$, the DEC which states $ v$ is future pointing causal iff $H-\Lambda\geq|1+\Lambda|>0$ holds.
		For any timelike $X$, the SEC $ (T^{spin}-\tr(T^{spin})g)(X,X)\geq0$ holds iff $1+\Lambda\geq0,\ H-1-2\Lambda\geq0$ or $1+\Lambda\geq0,\ H-1-2\Lambda<0,\ 2(1+\Lambda)\geq|H-1-2\Lambda|$ holds.
		Then we apply these conditions of $H$ and $\Lambda$ for each of solutions in type (i),(ii) and (iii), which implies the claim.
	\end{proof}
	
	Although the SqK-spinors of types (i) and (ii) are mathematically similar, this results show that they are quite different in nature if one is in a position to consider the energy conditions of classical spinor field.
	At the end of this section, we see that a SqK-spinor also gives a solution to the DM and EDM system as the following.
	
	\begin{proposition}\label{pr:DM and EDM}
		Let $(M^{r,3-r},\phi,\xi,\eta,g)$ be a three-dimensional pseudo-Riemannian Sasakian space-form with the $\phi$-sectional curvautre $H$, and let $\psi$ be a SqK-spinor of type $(a,b)=\left(\frac{(-1)^r}{2},\frac{H-(-1)^r}{4}\right)$ which satisfies $\xi\psi=\pm(-\sqrt{-1})^{r-1}\psi$ and has charge $1$, and let $\sqrt{-1}B\eta$ be a $U(1)$-gauge field, where $B\in\mathbb{R}$.
		If $4B=\pm(-1)^{r-1}\langle\psi,\psi\rangle$ holds, then the pair $(\psi,\sqrt{-1}B\eta)$ is a solution to the DM system.
		Furthermore, the triple $(\psi,\sqrt{-1}B\eta,g)$ is a solution to the EDM system with a cosmological constant $\Lambda$ for
		\begin{align}
			H &= \frac{-(-1)^r \langle\psi,\psi\rangle^2 +(-1)^r \langle\psi,\psi\rangle - 8}{\langle\psi,\psi\rangle -(-1)^r 8}, \\
			\Lambda &=  \frac{\langle\psi,\psi\rangle^{2}}{8} - \frac{\langle\psi,\psi\rangle}{4} + (-1)^r.
		\end{align}
	\end{proposition}
	
	\begin{proof}
		Since the spin-c connection for $\psi$ is given by $ \nabla^c\psi:=\nabla\psi+ \sqrt{-1}B\eta \psi$, the SqK-spinor satisfies $ \nabla^c_X\psi=(\sqrt{-1})^raX\psi+\sqrt{-1}(\pm b+B)\eta(X)\psi$.
		Then we obtain $ D^c\psi=(-1)^{r-1}(3a+b\pm B)\psi$.
		
		We have $*d*F=4B\eta$ and ${}^\flat J_\psi=\pm(-1)^{r-1}\langle\psi,\psi\rangle\eta$, then the Maxwell equation is given by
		\begin{align}
			4B=\pm(-1)^{r-1}\langle\psi,\psi\rangle. \label{eq:maxwell in EDM}
		\end{align}
		Therefore, the pair $(\psi,\sqrt{-1}B\eta)$ is a solution to the DM system if \eqref{eq:maxwell in EDM} holds.
		
		Next, the energy momentum tensors are given by
		\begin{align}
			T^{spin}&=-\frac{(-1)^r}{2}a\langle\psi,\psi\rangle g-\frac{b\pm B}{2}\langle\psi,\psi\rangle\eta\otimes\eta,\\
			T^{em}&=2B^2g+(-1)^{r-1}4B^2\eta\otimes\eta,
		\end{align}
		and thus, for $T=T^{spin}+T^{em}$, we have
		\begin{align}
			T-\tr(T)g+2\Lambda g=&\left\{\frac{(-1)^{r+1}}{2}(2a+b\pm B)\langle\psi,\psi\rangle + 2\Lambda\right\}g \nonumber\\
			&+ \left\{ 4B^2(-1)^{r+1}-\frac{1}{2}(b\pm B)\langle\psi,\psi\rangle \right\}\eta\otimes\eta
		\end{align}
		From ${\rm Ric}=(H+(-1)^r)g +(1-(-1)^rH)\eta\otimes\eta$, the Einstein equations are given by
		\begin{align}
			&H + (-1)^r = \frac{(-1)^r}{2}(2a+b\pm B)\langle\psi,\psi\rangle + 2\Lambda, \\
			&1 - (-1)^rH = 4B^2(-1)^{r+1}-\frac{1}{2}(b\pm B)\langle\psi,\psi\rangle. 
		\end{align}
		Then we substitute $a=\frac{(-1)^r}{2},\ b=\frac{H-(-1)^r}{4}$ into them, and solving the equations implies the claim.
	\end{proof}

	In the above Lorentzian case, we have
	\begin{align}
		H=-1+\frac{\langle\psi,\psi\rangle^2}{\langle\psi,\psi\rangle+8},
	\end{align}
	and then $M^{1,2}$ is a Lorentzian Sasakian space-form of type $S^3$ if $\langle\psi,\psi\rangle<-4,\ 8<\langle\psi,\psi\rangle$, and type $\widetilde{SL(2,\mathbb{R})}$ if $-4<\langle\psi,\psi\rangle<8$, and Nil$^{1,2}$ if $\langle\psi,\psi\rangle=-4,8$.

	\section{Explicit formulae of SqK-spinors}\label{sec:Explicit formula of SqK-spinors}
	We derive the explicit representations of SqK-spinors.
	The metric of the Berger sphere with $\phi$-sectional curvature $H$ is expressed as, for $\alpha\in\mathbb{R}_+$,
	\begin{align}
		g=\frac{\alpha}{4}(d\theta^2+\sin^2\theta d\varphi^2)+\frac{\alpha^2}{4}(d\phi+\cos\theta d\varphi)^2,\ H=\frac{4}{\alpha}-3,
	\end{align}
	using a Hopf-like coordinate $\{\theta,\varphi,\phi\}$.
	Then we define a Sasakian frame as
	\begin{align}
		e_1&=\frac{2}{\alpha}\partial_\phi=\xi,\label{eq:sasakian frame 1}\\
		e_2&=\frac{2}{\sqrt{\alpha}}\left(-\sin\phi\partial_\theta+\frac{\cos\phi}{\sin\theta}\partial_\varphi-\cot\theta\cos\phi\partial_\phi\right),\label{eq:sasakian frame 2}\\
		e_3&=\frac{2}{\sqrt{\alpha}}\left(\cos\phi\partial_\theta+\frac{\sin\phi}{\sin\theta}\partial_\varphi-\cot\theta\sin\phi\partial_\phi\right).\label{eq:sasakian frame 3}
	\end{align}
	
	Similarly the metric of a Lorentzian-Sasakian space-form of type  $\widetilde{SL(2,\mathbb{R})}$ with $\phi$-sectional curvature $  H$ is expressed as, for $\alpha\in\mathbb{R}_+$,
	\begin{align}
		g=\frac{\alpha}{4}(d\chi^2+\sinh^2\chi d\varphi^2)-\frac{\alpha^2}{4}(d\tau+\cosh\chi d\varphi)^2,
		H=-\frac{4}{\alpha}+3,
	\end{align}
	using the appropriate coordinate $\{\chi,\varphi,\tau\}$.
	Then we define a Sasakian frame as
	\begin{align}
		e_1&=\frac{2}{\alpha}\partial_\tau=\xi \label{eq:sasakian frame lorentz 1}\\
		e_2&=\frac{2}{\sqrt{\alpha}}\left(\cos\tau\partial_\chi+\frac{\sin\tau}{\sinh\chi}\partial_\varphi-\coth\chi\sin\tau\partial_\tau\right) \label{eq:sasakian frame lorentz 2}\\
		e_3&=\frac{2}{\sqrt{\alpha}}\left(-\sin\tau\partial_\chi+\frac{\cos\tau}{\sinh\chi}\partial_\varphi-\coth\chi\cos\tau\partial_\tau\right) \label{eq:sasakian frame lorentz 3}
	\end{align}
	
	The defining equations of a SqK-spinor are given by
	\begin{align}
		&e_1(\psi)+\left((\sqrt{-1})^r\frac{H-(-1)^r}{4}-(\sqrt{-1})^r(a+b)\right)\gamma_1\psi=0,\label{eq:sqk-spinor eq1}\\
		&e_2(\psi)+\left(\frac{(-\sqrt{-1})^r}{2}-(\sqrt{-1})^ra\right)\gamma_2\psi=0,\label{eq:sqk-spinor eq2}\\
		&e_3(\psi)+\left(\frac{(-\sqrt{-1})^r}{2}-(\sqrt{-1})^ra\right)\gamma_3\psi=0\label{eq:sqk-spinor eq3},
	\end{align}
	and we solve these SqK-equations in the following.
	
	\begin{proposition}\label{pr:explicit formula}
		Let $(M^{0,3},\phi,\xi,\eta,g)$ be a three-dimensional Riemannian Sasakian space-form with $\phi$-sectional curvautre $H$.
		Then a SqK-spinor of type $ \left(\frac{1}{2},\frac{H-1}{4}\right)$ is given by
		\begin{align}
			\psi=\begin{pmatrix}C_1 \\ C_2\end{pmatrix},\ C_1,C_2\in\mathbb{C},
		\end{align}
		with respect to an arbitrary Sasakian frame.
		In particular, on Berger sphere, that is $ H>-3$, SqK-spinors $\psi^\pm$ of type $ \left(\frac{1\pm \sqrt{3+H}}{2},\frac{-2\mp\sqrt{3+H}}{2}\right)$ is given by 
		\begin{align}
			&\psi^-= \begin{pmatrix} e^{\frac{\sqrt{-1} }{2}\phi} &  e^{- \frac{\sqrt{-1} }{2}\phi} \\   -e^{\frac{\sqrt{-1} }{2}\phi} & e^{- \frac{\sqrt{-1} }{2}\phi}\end{pmatrix}\begin{pmatrix} \cos{\frac{\theta}{2}} & -\sqrt{-1}\sin{\frac{\theta}{2}} \\ -\sqrt{-1}\sin{\frac{\theta}{2}} & \cos{\frac{\theta}{2}} \end{pmatrix}\begin{pmatrix} e^{\frac{\sqrt{-1}}{2}\varphi} & 0 \\ 0 & e^{-\frac{\sqrt{-1}}{2}\varphi} \end{pmatrix}\begin{pmatrix} C_1 \\ C_2 \end{pmatrix},\\ 
			&\psi^+=\xi\psi^-,\ \langle\psi^\pm,\psi^\pm\rangle=2(|C_1|^2+|C_2|^2),\ C_1,C_2\in\mathbb{C},
		\end{align}
		with respect to the Sasakian frame \eqref{eq:sasakian frame 1},\eqref{eq:sasakian frame 2} and \eqref{eq:sasakian frame 3}.

		Let $(M^{1,2},\phi,\xi,\eta,g)$ be a three-dimensional Lorentzian-Sasakian space-form with $\phi$-sectional curvature $  H$.
		Then a SqK-spinor of type $ \left(-\frac{1}{2},\frac{  H+1}{4}\right)$ is given by
		\begin{align}
			\psi=\begin{pmatrix}C_1 \\ C_2\end{pmatrix},\ C_1,C_2\in\mathbb{C},
		\end{align}
		with respect to an arbitrary Sasakian frame.
		In particular, in the case of $   H<3$, SqK-spinors $\psi^\pm$ of type $\left(\frac{-1 \pm \sqrt{-  H+3}}{2}, \frac{2 \mp \sqrt{-  H+3}}{2} \right)$ is given by 
		\begin{align}
			&\psi^+= \begin{pmatrix} e^{\frac{\sqrt{-1} }{2}\tau} &  e^{- \frac{\sqrt{-1} }{2}\tau} \\   e^{\frac{\sqrt{-1} }{2}\tau} & -e^{- \frac{\sqrt{-1} }{2}\tau}\end{pmatrix}\begin{pmatrix} \cosh{\frac{\chi}{2}} & -\sqrt{-1}\sinh{\frac{\chi}{2}} \\ \sqrt{-1}\sinh{\frac{\chi}{2}} & \cosh{\frac{\chi}{2}} \end{pmatrix}\begin{pmatrix} e^{\frac{\sqrt{-1}}{2}\varphi} & 0 \\ 0 & e^{-\frac{\sqrt{-1}}{2}\varphi} \end{pmatrix}\begin{pmatrix} C_1 \\ C_2 \end{pmatrix},\\
			&\psi^-=\xi\psi^+,\ \langle\psi^\pm,\psi^\pm\rangle=2(|C_1|^2-|C_2|^2),\ C_1,C_2\in\mathbb{C},
		\end{align}
		with respect to the Sasakian frame \eqref{eq:sasakian frame lorentz 1},\eqref{eq:sasakian frame lorentz 2} and \eqref{eq:sasakian frame lorentz 3}.
	\end{proposition}

	\begin{proof}
		We prove only the Riemannian case.
		Since the spin connection form \eqref{eq:spin conn form}, it is obvious that a spinor $\psi$ is a SqK-spinor of type $ \left(\frac{1}{2},\frac{H-1}{4}\right)$ if and only if $\psi$ is given by
		\begin{align}
			\psi=\begin{pmatrix}C_1 \\ C_2\end{pmatrix},\ C_1,C_2\in\mathbb{C},
		\end{align}
		with respect to a Sasakian frame.
		
		Next we consider the SqK-equations for $ (a,b)=\left(\frac{1- \sqrt{3+H}}{2},\frac{-2+\sqrt{3+H}}{2}\right)$ on Berger sphere.
		For the Sasakian frame \eqref{eq:sasakian frame 1},\eqref{eq:sasakian frame 2},\eqref{eq:sasakian frame 3}, we put $\hat e_1=\alpha e_1,\ \hat e_2=\sqrt{\alpha} e_2,\ \hat e_3=\sqrt{\alpha} e_3$, and
		\begin{align}
			\psi=\begin{pmatrix}f(\theta,\varphi,\phi) \\ h(\theta,\varphi,\phi)\end{pmatrix}.
		\end{align}
		Then the SqK-equations \eqref{eq:sqk-spinor eq1},\eqref{eq:sqk-spinor eq2},\eqref{eq:sqk-spinor eq3} are given by $\hat e_i(\psi)+\gamma_i\psi=0,\ (i=1,2,3)$, then these equations are written down as follows
		\begin{align}
			&\sqrt{-1} h{\left(\theta,\varphi,\phi \right)} + 2 \frac{\partial}{\partial \phi} f{\left(\theta,\varphi,\phi \right)}=0,\label{eq:sqks eq 1}\\
			&\sqrt{-1} f{\left(\theta,\varphi,\phi \right)} + 2 \frac{\partial}{\partial \phi} h{\left(\theta,\varphi,\phi \right)}=0,\label{eq:sqks eq 2}\\
			&\frac{h{\left(\theta,\varphi,\phi \right)}}{2} - \sin{\left(\phi \right)} \frac{\partial}{\partial \theta} f{\left(\theta,\varphi,\phi \right)} - \cos{\left(\phi \right)} \cot{\left(\theta \right)} \frac{\partial}{\partial \phi} f{\left(\theta,\varphi,\phi \right)} + \frac{\cos{\left(\phi \right)} \frac{\partial}{\partial \varphi} f{\left(\theta,\varphi,\phi \right)}}{\sin{\left(\theta \right)}}=0,\label{eq:sqks eq 3}\\
			&\frac{f{\left(\theta,\varphi,\phi \right)}}{2}
			+ \sin{\left(\phi \right)} \frac{\partial}{\partial \theta} h{\left(\theta,\varphi,\phi \right)}
			+ \cos{\left(\phi \right)} \cot{\left(\theta \right)} \frac{\partial}{\partial \phi} h{\left(\theta,\varphi,\phi \right)}
			- \frac{\cos{\left(\phi \right)} \frac{\partial}{\partial \varphi} h{\left(\theta,\varphi,\phi \right)}}{\sin{\left(\theta \right)}}=0,\label{eq:sqks eq 4}\\
			&\frac{\sqrt{-1} f{\left(\theta,\varphi,\phi \right)}}{2} - \sin{\left(\phi \right)} \cot{\left(\theta \right)} \frac{\partial}{\partial \phi} f{\left(\theta,\varphi,\phi \right)} + \frac{\sin{\left(\phi \right)} \frac{\partial}{\partial \varphi} f{\left(\theta,\varphi,\phi \right)}}{\sin{\left(\theta \right)}} + \cos{\left(\phi \right)} \frac{\partial}{\partial \theta} f{\left(\theta,\varphi,\phi \right)}=0,\label{eq:sqks eq 5}\\
			&\frac{\sqrt{-1} h{\left(\theta,\varphi,\phi \right)}}{2} 
			+ \sin{\left(\phi \right)} \cot{\left(\theta \right)} \frac{\partial}{\partial \phi} h{\left(\theta,\varphi,\phi \right)}
			- \frac{\sin{\left(\phi \right)} \frac{\partial}{\partial \varphi} h{\left(\theta,\varphi,\phi \right)}}{\sin{\left(\theta \right)}}
			- \cos{\left(\phi \right)} \frac{\partial}{\partial \theta} h{\left(\theta,\varphi,\phi \right)}=0,\label{eq:sqks eq 6}
		\end{align}
		
		From \eqref{eq:sqks eq 1},\eqref{eq:sqks eq 2}, we have
		\begin{align}
			f(\theta,\varphi,\phi)&=f_{1}{\left(\theta,\varphi \right)} e^{\frac{\sqrt{-1} \phi}{2}} + f_{2}{\left(\theta,\varphi \right)} e^{- \frac{\sqrt{-1} \phi}{2}},\\
			h(\theta,\varphi,\phi)&=- f_{1}{\left(\theta,\varphi \right)} e^{\frac{\sqrt{-1} \phi}{2}} + f_{2}{\left(\theta,\varphi \right)} e^{- \frac{\sqrt{-1} \phi}{2}}.
		\end{align}
		We substitute these into \eqref{eq:sqks eq 3},\eqref{eq:sqks eq 4},\eqref{eq:sqks eq 5},\eqref{eq:sqks eq 6}, then we have
		\begin{align}
			\partial_\theta\begin{pmatrix} f_1 \\ f_2 \end{pmatrix} &= -\frac{\sqrt{-1}}{2}\begin{pmatrix}   0 & 1 \\ 1 & 0\end{pmatrix}\begin{pmatrix} f_1 \\ f_2 \end{pmatrix},\label{eq:sqks eq 7}\\
			\partial_\varphi\begin{pmatrix} f_1 \\ f_2 \end{pmatrix} &= \frac{1}{2}\begin{pmatrix}   \sqrt{-1}\cos\theta & -\sin\theta \\ \sin\theta & -\sqrt{-1}\cos\theta\end{pmatrix}\begin{pmatrix} f_1 \\ f_2 \end{pmatrix}\label{eq:sqks eq 8}.
		\end{align}
		
		Finally, from \eqref{eq:sqks eq 7},\eqref{eq:sqks eq 8}, we obtain
		\begin{align}
			\begin{pmatrix} f_1 \\ f_2\end{pmatrix}=
			\begin{pmatrix}\cos{\frac{\theta}{2}}e^{\frac{\sqrt{-1}}{2}\varphi} & -\sqrt{-1}\sin{\frac{\theta}{2}}e^{-\frac{\sqrt{-1}}{2}\varphi} \\ -\sqrt{-1}\sin{\frac{\theta}{2}}e^{\frac{\sqrt{-1}}{2}\varphi} & \cos{\frac{\theta}{2}}e^{-\frac{\sqrt{-1}}{2}\varphi}  \end{pmatrix}\begin{pmatrix} C_1 \\ C_2\end{pmatrix}.
		\end{align}
		
	\end{proof}

	\section{Conclusions and Discussions}\label{sec:Conclusions}
	In this paper we studied in detail the properties of a SqK-spinor on three-dimensional pseudo-Riemannian Sasakian manifolds.
	We summarize the obtained results.
	The SqK-spinors on the three-dimensional Lorentzian-Sasakian manifold was completely classified, similar to \cite{Kim Friedrich 2000}.
	Then we clarify the following geometric properties of a SqK-spinor.
	We have shown that the Reeb vector field defines a mapping. 
	The mapping maps a SqK-spinor to another one.
	We have also demonstrated that the Dirac current of a SqK-spinor describes the motion of a charged particle in the presence of a contact Maxwell field.
	We identified the necessary and sufficient condition for the Dirac current of a SqK-spinor, that is not a Killing spinor, to be a Killing vector field.
	We proved that two SqK-spinors can be used to construct a Maxwell system with a source.
	Next, we discussed that a SqK-spinor solves to the ED, DM and EDM system.
	A WK-spinor was completely determined in three-dimensional pseudo-Riemannian Sasakian space-form by the same argument as in \cite{Kim Friedrich 2000}, and then we constructed solutions to the ED system.
	We have also constructed solutions to the DM and EDM system with a SqK-spinor and a contact Maxwell field.
	The constructed 1+2 dimensional solutions are interesting from general relativity viewpoint.
	Finally, we have given an explicit representation of SqK-spinors.
	
	We highlight the significance of each of our results.
	Proposition \ref{pr:killing vector} shows that the Reeb vector field of a three-dimensional pseudo-Riemannian Sasakian space-form is described by a SqK-spinor in $\mathcal{S}_0$.
	This result is of interest in both Sasakian and spin geometry.
	
	In \cite{Cabrerizo 2009}, on contact metric manifolds, a contact magnetic field whose gauge potential is given by a contact form was proposed.
	The motion of charged particles in the presence of a contact magnetic field is a subject of interest and has been studied to some extent.
	The proposition \ref{pr:magnetic curve} shows that the motion of charged particles in the presence of contact Maxwell field in a three-dimensional pseudo-Riemannian Sasakian space-form can be described by a SqK-spinor under suitable initial conditions.
	It is interesting that the motion of a charged particle is realized as a Dirac current of a particular spinor.
	
	For a Killing spinor $\psi$ on the round $S^3$, the Dirac current $J_\psi$ and the two-form $F_{\psi,\psi}$ satisfy the Maxwell equation with source \cite{Fujii 1886}.
	The current $J_\psi$ therefore generates a magnetic field $F_{\psi,\psi}$.
	The magnetic field $F_{\psi,\psi}$ is the contact magnetic field.
	A naive generalization of this result is the proposition \ref{pr:maxwell with source}, which states that for two SqK-spinors $\psi_1$ and $\psi_2$, the resulting current $(4(a_1+a_2)+2(b_1+b_2))\Re(J_{\psi_1,\psi_1})-2(b_1+ b_2)g(J,e_1)e_1$ generates the Maxwell field $\Re(F_{\psi_1,\psi_2})$.
	
	In \cite{Kim Friedrich 2000}, a SqK-spinor was introduced as a candidate for a WK-spinor, and then it was used to construct solutions to the ED system.
	We showed that almost all SqK-spinors are WK-spinors with non-zero cosmological constant, and we constructed  six distinct solutions to the ED system.
	Furthermore, for the Lorentzian solutions, we clarify the range of $\phi$-sectional curvature for which the energy conditions are satisfied.
	This demonstrates that the SqK-spinors in types (i) and (ii) of Proposition \ref{pr:energy condition} exhibit quite different physical property from an energy condition view point although they are mathematically similar.
	
	We also found that a SqK-spinor solves to the EDM system with a contact Maxwell field.
	In this system, the SqK-spinor generates a Dirac current proportional to the Reeb vector field, and the Dirac current is the source of the contact Maxwell field, which in turn is the source of the Einstein gravity.
	Although it is generally difficult to construct exact solutions to an EDM system because of the interaction between the three elements, the gravity field, a spinor and a Maxwell field, we constructed a family of exact solutions.
	The total energy-momentum tensor is expressed only in terms of $\langle\psi,\psi\rangle$ and we know exactly how the curvature of spacetime varies with the configuration of matter.
	This is significant as a toy model.
	In the limit of $\langle\psi,\psi\rangle\to0$ for this family of solutions, the contact Maxwell field and the spinor field vanish and the spacetime becomes $AdS_3$ which satisfies $H=-1$.
	When matter enters into $AdS_3$, the spatial curvature increases, therefore $-1<H$ holds.
	
	Although the existence of SqK-spinors on three-dimensional pseudo-Riemannian Sasakian space-form has been shown, an explicit formula of a SqK-spinor but not Killing spinor has not been previously known, and we found formulas expressed using elementary functions with respect to a Sasakian frame.
	In \cite{Fujii 1886}, an explicit representation of a Killing spinor on the round $S^3$ was given.
	A SqK-spinor of type $(\frac{1}{2},\frac{H-1}{4})$ is a Killing spinor of round $S^3$ if $H=1$.
	According to the Proposition \ref{pr:explicit formula} therefore the Killing spinor of the round $S^3$ is the spinor whose components are all constant with respect to a Sasakian frame.
    This is a simple representation.
	An explicit representation of a SqK-spinor is also significant from a general relativistic perspective for constructing solutions to the EM and the EDM system.

	\section*{Acknowledgement}
	This work was partly supported by MEXT Promotion of Distinctive Joint Research Center Program JPMXP0723833165.
	The data that supports the findings of this study are available within the article and its supplementary material.

\end{document}